\RequirePackage{fix-cm}
\documentclass[11pt, a4paper, oneside, DIV11, final]{amsart}
\usepackage[notref,notcite]{showkeys}
\usepackage[latin1]{inputenc}
\usepackage[T1]{fontenc}
\usepackage{amssymb}
\usepackage[stretch=10]{microtype}
\usepackage{mathtools}
\usepackage[DIV11, headinclude=true]{typearea}
\usepackage[hidelinks, linktoc=all]{hyperref}

\theoremstyle{plain}
\newtheorem{thm}{Theorem}[section]
\theoremstyle{remark}
\newtheorem{rem}[thm]{Remark}
\newtheorem*{rem*}{Remark}
\theoremstyle{plain}
\newtheorem{cor}[thm]{Corollary}
\theoremstyle{plain}
\newtheorem{lem}[thm]{Lemma}

\numberwithin{equation}{section}

\DeclareMathOperator{\Ext}{Ext}
\DeclareMathOperator{\Sym}{Sym}
\DeclareMathOperator{\Pos}{Pos}
\DeclareMathOperator{\Tr}{Tr}
\DeclareMathOperator{\uppdensity}{\overline{Dens}}
\DeclarePairedDelimiter{\norm}{\lVert}{\rVert}
\DeclarePairedDelimiter{\abs}{\lvert}{\rvert}
\DeclarePairedDelimiter{\Card}{\lvert}{\rvert}
\newcommand{\st}{\::\:}

\newcommand{\dd}{\mathop{}\!\mathrm{d}}

\begin{document}

\title{Subadditive and multiplicative ergodic theorems}

\author{S\'ebastien Gou\"ezel and Anders Karlsson}
\thanks{Supported in part by the Swiss NSF grant 200021 132528/1.}

\address{
Laboratoire Jean Leray, CNRS UMR 6629, Universit\'e de Nantes, 2 rue de la
Houssini\`ere, 44322 Nantes, France} \email{sebastien.gouezel@univ-nantes.fr}

\address{Section de math\'ematiques,
Universit\'e de Gen\`eve, 2-4 Rue du Li\`evre, Case Postale 64, 1211 Gen\`eve 4,
Suisse} \email{anders.karlsson@unige.ch}

\address{Matematiska institutionen, Uppsala universitet, Box 256, 751 05
Uppsala, Sweden} \email{anders.karlsson@math.uu.se}

\begin{abstract}
A result for subadditive ergodic cocycles is proved that provides more
delicate information than Kingman's subadditive ergodic theorem. As an
application we deduce a multiplicative ergodic theorem generalizing an
earlier result of Karlsson-Ledrappier, showing that the growth of a random
product of semi-contractions is always directed by some horofunction. We
discuss applications of this result to ergodic cocycles of bounded linear
operators, holomorphic maps and topical operators, as well as a random mean
ergodic theorem.
\end{abstract}

\maketitle

\section{Introduction}

Products of random operations arise naturally in a variety of contexts from
pure mathematics to more applied sciences. Typically the operations, or the
maps, do not commute, but one would nevertheless hope to have asymptotic
regularity of various associated quantities. In the commutative case one has
the standard ergodic theorem or what in probability is called the law of
large numbers. A very important and genuinely non-commutative case is that of
products of random matrices. These are governed by the multiplicative ergodic
theorem of Oseledets which in particular is a fundamental theorem in
differentiable dynamics. Another area which belongs to this study is the
subject of random walks on groups.

It is a remarkable fact that, in many such situations, one can introduce a
metric which is invariant or non-expanded by the transformations under
consideration. This gives a way to quantify the behaviour of random products
of maps, such as linear operators, holomorphic maps, symplectomorphisms, or
homogeneous-monotone maps. Due to the non-expansion of the metric and the
triangle inequality, numerical quantities associated to the random products
then satisfy a form of subadditivity.

Kingman proved in~\cite{Ki68} the subadditive ergodic theorem, which is a
generalization of the usual ergodic theorem to subadditive cocycles. This
extension is very useful, with many applications. In the particular case of
random products,  Kingman's theorem asserts that there is a well-defined growth rate,
or in a different terminology, a certain speed with which these products tend to infinity.

Our goal is to understand to what extent random products tend to infinity
following a specific direction, using the notion of horofunction.
Horofunctions made one of their first explicit appearances in the 1926
Wolff-Denjoy theorem which describes the dynamics of holomorphic self-maps of
the unit disk. As was noted by those authors, and extended and commented on
by several people since then, the mechanism behind this result is the
Schwartz lemma which implies that the holomorphic maps do not increase the
Poincar\'e distance between points, and the fact that the unit disk with this
metric is the hyperbolic plane, see e.g.~\cite{K01,KeL07,AR14}.

Our strategy is to show first a substantial refinement of Kingman's theorem.
Then, we apply it to prove a very general multiplicative ergodic theorem,
extending one aspect of the Wolff-Denjoy theorem to a vastly more general
setting: the asymptotic behaviour of random products of 1-Lipschitz maps of
any metric space in terms of horofunctions. This generalizes and reproves the
main theorem in~\cite{KL06}, which in turn extends several known results,
such as the one of Oseledets mentioned above, and which has delivered unexpected
applications. Our theorem has a weak-type formulation involving linear or
metric functionals (a generalization of horofunctions to a non-proper
setting). Hence, it can hold in a very general setting, contrary to results
yielding a stronger convergence, that are known to fail if the geometric
properties of the space are not good enough. Moreover, under suitable
assumptions on the space, our a priori weaker statement can automatically be
promoted to the stronger one.

As for further new applications, our theory leads to an ergodic theorem for
cocycles of bounded linear operators. Ruelle proved the first such theorem in
infinite dimension assuming compactness of the operators~\cite{R82}. It was
generalized by Ma\~n\'e, Thieullen and others, see the recent
monograph~\cite{LL10} for more details. The interest for such statements
about semi-flows on Hilbert spaces comes from Ruelle to the more recent work
of Lian and Young in the study of certain stochastic differential equations,
partial differential equations of evolution with application to hydrodynamic
turbulence, such as the Navier-Stokes equation~\cite{R82,R84,ER85,LY12}.
There are also other potential contexts of application, see for example the
remark in~\cite[p.~10]{F02}.

Our subadditive theorem is stated in Paragraph~\ref{subsec:new_kingman} and
proved in Section~\ref{sec:Subadditive-ergodic-cocycles}. The ergodic theorem
for random products is then stated in Paragraph~\ref{subsec:ergodic_product}
and proved in Section~\ref{sec:Application-to-multiplicative}. Finally,
Paragraph~\ref{subsec:applications} and the remaining sections of this
article are devoted to a brief discussion of some applications.

\subsection{Existence of good times for subadditive cocycles}
\label{subsec:new_kingman}

Let $(\Omega,\mu)$ be a measure space with $\mu(\Omega)=1$ and let
$T:\Omega\rightarrow\Omega$ be an ergodic measure preserving map. A
measurable function $a:\mathbb{N}\times\Omega\rightarrow\mathbb{R}$ which
satisfies\[ a(n+m,\omega)\leq a(n,\omega)+a(m,T^{n}\omega)\] for all integers
$n,m>0$ and a.e.\ $\omega\in\Omega$ is called a \emph{subadditive cocycle.
}For convenience we also set $a(0,\omega)\equiv0$. One says that $a$ is
\emph{integrable} if $a(1,\omega)$ is integrable and one defines the
\emph{asymptotic average}
\[
  A=\inf_{n}\frac{1}{n}\int_{\Omega}a(n,\omega)\dd\mu \in [-\infty, +\infty).
\]
Kingman's theorem asserts that, almost surely, $a(n,\omega)/n$ tends to $A$.
Moreover, if $A>-\infty$, the convergence also holds in $L^1(\mu)$.

\medskip

We prove in Section~\ref{sec:Subadditive-ergodic-cocycles} the following
subadditive ergodic statement (\emph{cf}.\ Problem 3.3 in~\cite{K02}):
\begin{thm}
\label{thm:subadd} Let $a(n,\omega)$ be an integrable and subadditive cocycle
relative to the ergodic system $(\Omega,\mu,T)$ as above, with finite
asymptotic average $A$. Then for almost every $\omega$ there are integers
$n_{i}\coloneqq n_{i}(\omega)\rightarrow\infty$ and positive real numbers
$\delta_{\ell}\coloneqq \delta_{\ell}(\omega)\rightarrow0$ such that for
every~$i$ and every $\ell\leq n_{i}$,
\begin{equation}
  \label{eq:subadd}
  \abs{a(n_{i},\omega)-a(n_{i}-\ell,T^{\ell}\omega)-A\ell}
  \leq
  \ell \delta_{\ell}(\omega).
\end{equation}
\end{thm}

This statement significantly refines Proposition~4.2 in~\cite{KM99}. It is
not a consequence of Kingman's theorem: by subadditivity, we have
\begin{equation*}
  a(n_{i},\omega)-a(n_{i}-\ell,T^{\ell}\omega) \leq a(\ell, \omega) \sim A \ell,
\end{equation*}
by Kingman's theorem, so the upper bound in~\eqref{eq:subadd} readily follows
from Kingman's theorem. On the other hand, the lower bound, which asserts
that the cocycle is close to being additive at \emph{all} times $\ell$
between $0$ and $n_i$, is much more delicate.

This lower bound is reminiscent of Pliss' lemma, a combinatorial lemma which
proved very useful in hyperbolic dynamics, see for
instance~\cite{ABV00}. For any additive sequence tending linearly to
infinity, this lemma entails the existence of good times $n_i$ for which the
behavior of the sequence between $n_i-\ell$ and $n_i$ is well controlled for
all $\ell \leq n_i$. Our statement is both weaker (since there is an
additional error $\ell \delta_{\ell}$) and stronger since it applies in
random subadditive situations and makes it possible to approach the true
asymptotics $A \ell$.

We will apply Theorem~\ref{thm:subadd} to the context of multiplicative
ergodic theorems below, but it could also be of interest for different
questions, for example~\cite{GG14} used~\cite[Prop.~4.2]{KM99} to re-prove
and extend Livsic's theorem of~\cite{Ka11}.
\begin{rem}
\label{rmk:subadd} Define the upper asymptotic density of a subset $U$ of the
natural numbers as
\[
  \uppdensity(U)=\limsup_{N\rightarrow\infty}\Card{U\cap [0,N-1]}/N.
\]
The proof of Theorem~\ref{thm:subadd} gives in fact more information, namely
that on a set of large measure one can take $\delta_\ell$ independent of
$\omega$, and one can have many good times $n_i$. More precisely, fix
$\rho>0$, then there exist a sequence $\delta_{\ell}\rightarrow 0$ and a
subset $O\subset\Omega$ of measure at least $1-\rho$ such that, for every
$\omega\in O$, the subset of good times $A(\omega)\subset\mathbb{N}$ (made of
those $n$ for which~\eqref{eq:subadd} holds for all $\ell\leq n$) has upper
asymptotic density at least $1-\rho$.
\end{rem}

\subsection{Random products and metric functionals}
\label{subsec:ergodic_product}

Horocycles, horodisks etc are concepts originally coming from two-dimensional
hyperbolic geometry and complex analysis. A general definition of the
corresponding horofunctions (now also called Busemann functions) in terms of
geodesic rays $\gamma(t)$ was noted by Busemann:
\[
b_{\gamma}(\cdot)=\lim_{t\rightarrow\infty}d(\cdot,\gamma(t))-d(\gamma(0),\gamma(t)).
\]
As emphasized by Gromov, this definition leads to a natural bordification
of metric spaces, by mapping the space into its set of continuous
functions equipped with the topology of uniform convergence on bounded sets.
We consider instead pointwise convergence, following for example~\cite{GV12}.
Let $(X,d)$ be a metric space, fix $x_{0}\in X$, and define the continuous
injection
\begin{align*}
\Phi:X&\hookrightarrow\mathbb{R}^{X}\\
x & \mapsto h_{x}(\cdot)\coloneqq d(\cdot,x)-d(x_{0},x).
\end{align*}
The functions $h_x$ are all $1$-Lipschitz maps. As indicated, we endow the
space $\mathbb{R}^{X}$ of real valued functions on $X$ with the product
topology, i.e., the topology of pointwise convergence. The closure of the image
$ \overline{\Phi(X)}$ will therefore be compact.  By definition we call the elements in this compact set for
\emph{metric functionals}. Thus, to every point $x$ there is a unique
associated metric functional $h_x$, and then there may be further functionals
obtained as limit points:
\[
\widehat{X}\coloneqq \overline{\Phi(X)}\setminus  \Phi(X).
\]

To relate to the standard notion, we call limits, as $x\rightarrow\infty$, of $h_x$ in the uniform
convergence on bounded sets for  \emph{horofunctions}. 
If $X$ is proper and geodesic, then $\widehat{X}$ is precisely the set of
horofunctions.
If the metric space is particularly nice, for instance CAT(0), then
horofunctions and Busemann functions coincide. For non-proper metric spaces a
Busemann function might not be a horofunction since the convergence might not
be uniform on bounded sets, and conversely a horofunction might not be a
Busemann function since it might be obtained as a limit which does not
correspond to any geodesic ray. Our terminology is in part inspired by the
simple fact that for infinite dimensional Hilbert space $H$, the set
$\widehat{H}$ contains the closed unit ball of continuous linear functionals. 
The definition of metric functionals is also somewhat parallel to the one of linear functionals: metric functionals are maps 
$X\rightarrow \mathbb{R}$ that vanish at the origin $x_0$ and respect the metric structure of the spaces.

A map $f:X\rightarrow X$ is called \emph{non-expansive}, or
\emph{semi-contractive}, if
\[
  d(f(x),f(y))\leq d(x,y)
\]
for all $x,y\in X$. The set of all semi-contractive maps on $X$ is denoted by
$SC(X)$.

As in the previous paragraph, let $(\Omega,\mu)$ be a measure space with
$\mu(\Omega)=1$ and let $T:\Omega\rightarrow\Omega$ be an ergodic measure
preserving map. Given a map $\varphi:\Omega\rightarrow SC(X)$, one forms the
associated \emph{ergodic cocycle }
\begin{equation}
  \label{eq:def_u}
  u(n,\omega)=\varphi(\omega)\varphi(T\omega)\dotsm \varphi(T^{n-1}\omega).
\end{equation}
Note the order in which the maps are composed. We require a weak
measurability property: for all $x\in X$ and all $n \in \mathbb{N}$, the map
$\omega \mapsto u(n,\omega)x$ from $\Omega$ to $X$ should be measurable. For
instance, this is the case if $\varphi:\Omega \to SC(X)$ is measurable where
$SC(X)$ is endowed with the compact-open topology (i.e., the topology of
uniform convergence on compact subsets of $X$) and $X$ is locally compact
(this last assumption ensures that the composition map $SC(X) \times SC(X)
\to SC(X)$ is continuous, so that $u(n,\cdot):\Omega \to SC(X)$ is also
measurable).
This is also the case when $X$ is a Banach space and $\varphi$ is measurable
from $\Omega$ to the space of bounded linear operators on $X$ with the
topology of norm convergence.

We say that the above cocycle $u(n,\omega)$ is \emph{integrable} if
\[
\int_{\Omega}d(x,\varphi(\omega)x)\dd\mu<\infty,
\]
a condition which is independent of $x\in X$. In this case, the subadditive
cocycle $a(n,\omega)=d(x, u(n,\omega)x)$ is also integrable. Hence, by
Kingman's theorem, $d(x, u(n,\omega)x)/n$ converges almost surely to a limit
$A\geq 0$ (which does not depend on the choice of the basepoint $x$).

In Section~\ref{sec:Application-to-multiplicative}, the above subadditive
ergodic statement is used to establish the following multiplicative ergodic
theorem:
\begin{thm}
\label{thm:met}Let $u(n,\omega)$ be an integrable ergodic cocycle of
semi-contractions of a metric space $(X,d)$.
Then for a.e.\ $\omega$ there exists a metric functional
$h^{\omega}$ of $X$ such that for all $x$
\[
\lim_{n\rightarrow\infty}-\frac{1}{n}h^{\omega}(u(n,\omega)x)=\lim_{n\rightarrow\infty} \frac {1}{n}d(x, u(n,\omega)x).\]
Moreover, if $X$ is separable and $\Omega$ is a standard probability space,
one can choose the map $\omega \mapsto h^\omega$ to be Borel measurable.
\end{thm}
The main theorem of~\cite{KL06} is the same statement, but with the
additional assumption that the cocycle $u$ takes it values in isometries of
$X$, instead of semi-contractions (moreover it was formulated only for proper spaces).
This additional assumption makes it
possible to use the action of the cocycle on the space of metric functionals,
and use an additive cocycle there. This proof cannot work for
semi-contractions. The present proof will instead use
Theorem~\ref{thm:subadd} and is thus quite different.

Theorem~\ref{thm:met} was conjectured in~\cite{K04} for proper
metric spaces. With a use of the Hahn-Banach theorem, see
Section~\ref{sec:Application-to-multiplicative}, this specializes to the
following statement in the case of Banach spaces.
\begin{cor}
\label{cor:Banach}Let $u(n,\omega)$ be an integrable ergodic cocycle of
semi-contractions of a subset $D$ of a Banach space $X$. Then for a.e.\
$\omega$ there is a linear functional $f^{\omega}$ of norm 1 such that for
any $x\in D$,
\[
\lim_{n\rightarrow\infty}\frac{1}{n}f^{\omega}(u(n,\omega)x)=\lim_{n\rightarrow\infty}\frac{1}{n}\norm{
u(n,\omega)x} .
\]
\end{cor}
This generalizes the main theorem in Kohlberg-Neyman~\cite{KN81}, dealing
with $u(n,\omega)=A^{n}$ and $D$ convex (the latter condition was removed
in~\cite{PR83}). For the random setting previous results can be found
in~\cite{K04}. When $X$ is strictly convex and reflexive, the conclusion
implies weak convergence of $u(n,\omega)x/n$. When the norm of the dual of
$X$ is Fr\'echet differentiable, the conclusion implies norm convergence of
$u(n,\omega)x/n$. In general however the above statement is best possible in
view of a counterexample in~\cite{KN81}.

\subsection{Applications}
\label{subsec:applications}

In this paragraph, we describe briefly different settings where our results
apply. More involved applications are described in
Sections~\ref{sec:Cocylces-of-bounded}, \ref{sec:Cocycles-of-holomorphic}
and~\ref{sec:Behaviour-of-extremal}.

\medskip

Applied mathematics provides a wealth of examples of non-expansive mappings
of Banach spaces, especially $\ell^{\infty}$, for example in dynamical
programming and topical matrix multiplication (homogeneous-order-preserving),
generalizing matrices in the max-plus or min-plus (tropical) semi-ring. In
finite dimensions the existence of Lyapunov exponents has been studied, in
particular what could then be called a tropical Oseledets multiplicative
ergodic theorem. We refer to~\cite{CT80,Co88,Ma97,Gu03}. Our
Corollary~\ref{cor:Banach} implies some known and some not previously known
statements in this setting.

More precisely, let $S$ be a set and consider the Banach space $B(S)$ of
bounded real-valued functions $f:S\rightarrow\mathbb{R}$ with the sup-norm.
Consider a map $A:B(S)\rightarrow B(S)$ (not necessarily linear) with the
properties:
\begin{itemize}
\item (monotonicity) If $f(x)\leq g(x)$ for all $x$, then
    $(Af)(x)\leq(Ag)(x)$ for all $x$
\item (semi-homogeneity) for any positive constant $a$, it holds that
\[
A(f(\cdot)+a)(x)\leq Af(x)+a.
\]
\end{itemize}
Blackwell observed in 1965 that such maps $A$ are semi-contractive. He had a
constant $0<\beta<1$ in the second condition which corresponds to discounting
in financial settings. When $S=\left\{ 1,2,\dotsc,d\right\} $ these operators, which now can be viewed as functions
$A:\mathbb{R}^{d}\rightarrow\mathbb{R}^{d}$, and with equality in the second
condition, are sometimes called \emph{topical functions} (see~\cite{Gu03}).

A multiplicative variant of this type of maps are self-mappings of cones
$A:C\rightarrow C$ with $A(ax)=aA(x)$ and $x\leq y$ implies $Ax\leq Ay$ with
respect to the cone partial order. Such maps are semi-contractive in
Hilbert's metric and its variants.

Our theorems apply to random products of such mappings, and information about
metric functionals can be found in~\cite{Wa08}.

\medskip

Consider now the random ergodic set-up, first considered in particular by
von~Neumann-Ulam, see~\cite{Ka52}. Let $L_{\omega}$ be a collection of
measure-preserving transformations of a probability space $(Y,\nu)$, indexed
by $\omega \in \Omega$. Let $T:\Omega\to \Omega$ be an ergodic
measure-preserving map. Given $v:Y\to \mathbb{R}$ and $y\in Y$, $\omega\in
\Omega$, we consider the average
\[
  \frac{1}{n}\sum_{k=0}^{n-1}v(L_{T^{k-1}\omega}L_{T^{k-2}\omega}\dotsm L_{\omega}y).
\]
Introduce an isometry $\varphi(\omega)$ of the space $X=L^2(Y,\nu)$ by
$(\varphi(\omega) w)(y)=v(y) + w(L_\omega y)=v +U_\omega w$. Let
$u(n,\omega)$ denote the corresponding multiplicative cocycle. Then the above
average equals $(u(n,\omega)0)(y)/n$. Hence, the following corollary (which
follows readily from Corollary~\ref{cor:Banach}) is a generalization of the
(random) mean ergodic theorem:
\begin{cor}
\label{cor:randommeanergodic}
Let $X$ be a Banach space and let $U$ be a
strongly measurable map from $\Omega$ to linear operators on $X$. Suppose
that $\norm{ U_\omega} \leq 1$ for every $\omega\in \Omega$. Then for every
$v\in X$ and a.e.\ $\omega$ there is a linear functional $F_\omega$ of $X$
with $\norm{ F_\omega} =1$ such that
\[
\lim_{n\rightarrow\infty}\frac{1}{n}F_\omega
  \left(\sum_{k=0}^{n-1}U_{\omega}U_{T\omega}\dotsm U_{T^{k-1}\omega}v\right)
=\lim_{n\rightarrow\infty}\frac{1}{n}\norm*{
  \sum_{k=0}^{n-1}U_{\omega}U_{T\omega}\dotsm U_{T^{k-1}\omega}v}.
\]

\end{cor}
Again we remark that when $X$ is strictly convex and reflexive, the
conclusion implies weak convergence of the ergodic average in question and
when the norm of the dual of $X$ is Fr\'echet differentiable, the conclusion
implies norm convergence. It is however known, due to Yosida~\cite{Ka52},
that in this situation if there is a weakly convergent subsequence then
strong convergence of the whole sequence follows. So $X$ being strictly
convex and reflexive would suffice. One of the most general results of this
type, with norm convergence, was obtained by Beck-Schwartz for reflexive
spaces in 1957. (There are many papers considering norm convergence of
similar or more general averages.) Note here that it is well-known that the
Carleman-von~Neumann mean ergodic theorem (i.e.\ with the norm convergence)
does not hold in general for Banach spaces. Our statement does hold, and
implies as said norm convergence under the conditions mentioned.

\medskip

As another application we show in Section~\ref{sec:Cocylces-of-bounded}:
\begin{thm}
\label{thm:operators} Let $v(n,\omega)=A(T^{n-1}\omega)A(T^{n-2}\omega)\dotsm
A(\omega)$ be an integrable ergodic cocycle of bounded invertible linear
operators of a Hilbert space. Denote the positive part
$\left[v(n,\omega)\right] \coloneqq
\left(v(n,\omega)^{*}v(n,\omega)\right)^{1/2}$.
Then for a.e.\ $\omega$ there is a norm 1
linear functional $F_{\omega}$ on the space of bounded linear operators such that
\[
\lim_{n\rightarrow\infty}\frac{1}{n}F_{\omega}\left(\log\left[v(n,\omega)\right]\right)=\lim_{n\rightarrow\infty}\frac{1}{n}\norm{
\log\left[v(n,\omega)\right]} .
\]

\end{thm}
This general statement can under further assumptions be promoted to yield
some known theorems. The same remarks as after Corollary~\ref{cor:Banach}
apply. For example if the Hilbert space has finite dimension, then $\Sym$ can
be given a Hilbert space structure and it follows that
\[
  \frac{1}{n}\log\left[v(n,\omega)\right]
\]
converges. This is known to be equivalent to Oseledets' multiplicative
ergodic theorem, as for example explained in~\cite{R82}. In the infinite
dimensional case and restricting to identity plus Hilbert-Schmidt operators, one
gets a uniform convergence, as observed in~\cite{KM99}, stronger than that
in~\cite{R82} for compact operators. Note that in general the conclusion of
the standard multiplicative ergodic theorem is too strong: if $A$ is a
bounded linear operator on a Hilbert space and $v$ is a nonzero vector, it is
not true in general that
\[
  \frac{1}{n}\log\norm{ A^{n}v} ,
\]
converges as $n\rightarrow\infty$. See~\cite{LL10, Sc06} for comparison.

In Section~\ref{sec:Cocycles-of-holomorphic} we exemplify a consequence of
Theorem~\ref{thm:met} in the complex analytic setting, and yet another
application in Section~\ref{sec:Behaviour-of-extremal}. A further possibility
would be to consider random products of diffeomorphisms of a compact
manifold, and exploit the induced isometric action on the space of Riemannian
metrics. This space has been studied somewhat, it is known to be CAT(0) but
not complete. For two or more generic diffeomorphisms there is no joint
invariant measure on the manifold, so Oseledets' theorem does not apply, but
our Theorem~\ref{thm:met} applies.

\section{Subadditive ergodic cocycles}
\label{sec:Subadditive-ergodic-cocycles}

In this section, we prove Theorem~\ref{thm:subadd}. By passing to the natural
extension if necessary, one can assume without loss of generality that $T$ is
invertible.

One of the inequalities in~\eqref{eq:subadd} is known. Indeed, by definition
$a(n,\omega)\leq a(n-\ell,T^{\ell}\omega)+a(\ell,\omega)$ and by Kingman's
theorem $a(\ell,\omega)=A\ell+o(1)$. This means that\[ a(n,\omega)\leq
a(n-\ell,T^{\ell}\omega)+A\ell+\ell\delta_{\ell}\] where
$\delta_{\ell}\rightarrow0$.

The other inequality is much more subtle. It says intuitively that at certain
times the cocycle is nearly additive. Note that Theorem~\ref{thm:subadd} for
additive cocycles is equivalent to Birkhoff's ergodic theorem.

We begin by defining \[ b(n,\omega)=a(n,T^{-n}\omega).\] This is again a
subadditive cocycle but for the transformation $\tau\coloneqq T^{-1}$ and
with the same asymptotic average as $a$. The interest in $b$ comes from the
fact  that
\[
  a(n,\omega)-a(n-\ell,T^{\ell}\omega)=b(n,\omega')-b(n-\ell,\omega'),
\]
where $\omega'=T^{n}\omega$. Note that, on the right-hand side, the point
$\omega'$ is the same in both terms. We start by showing that
$b(n,\omega)-b(n-\ell,\omega)$ behaves well for the majority of times $n$ and
for $\ell\leq n$ large enough, by a combinatorial argument related to several
proofs of Kingman's theorem (for example as in Steele~\cite{S89}). This is
the key point of the proof, given in Lemma~\ref{lem:crucial}. In contrast
to~\cite{KM99}, we will control the majority of times and not just a small
set of times. This is central to keep a good control once one changes
variables back from $b$ to $a$.

We begin by proving the following lemma as a warm-up. These arguments will be
used again in a more elaborate form in the proof of Lemma~\ref{lem:crucial}.
\begin{lem}
\label{lem:warm-up}Let $b$ be an integrable ergodic subadditive cocycle with
finite asymptotic average $A$. Let $\delta>0$. Then there exists $c>0$ such
that for almost every~$\omega$
\[
\uppdensity\left\{ n\in\mathbb{N}\st \exists\ell\in\left[1,n\right],b(n,\omega)-b(n-\ell,\omega)\leq(A-c)\ell\right\} \leq\delta.
\]
\end{lem}
\begin{proof}
By replacing if necessary $b(n,\omega)$ by $b(n,\omega)-An$, one may without
loss of generality assume that $A=0$. We denote by $\tau$ the underlying
ergodic transformation. Let $V=\{n\in\mathbb{N}\st \exists\ell\in[1,n],\
b(n,\omega)-b(n-\ell,\omega)\leq-c\ell\}$. When $c$ is large enough we would
like to conclude that the density of $V$ is small.

Fix $N>0$. We will partition $[0,N]$ using the following algorithm. First let
$n_{0}=N$. Assuming $n_{i}$ is defined, we proceed as follows. If
$n_{i}\notin V$, then take $n_{i+1}=n_{i}-1$. If $n_{i}\in V$, then let
$\ell_{i}\in[1,n_{i}]$ be as in the definition of $V$, and let
$n_{i+1}=n_{i}-\ell_{i}$. We stop when $n_{i}=0$.

We have thus decomposed the interval $[0,N)$ in a union of intervals
$[n_{i}-1,n_{i})$ (with $n_{i}\notin V$), and $[n_{i}-\ell_{i},n_{i})$ (with
$n_{i}\in V$). Using the subadditivity along the intervals of the first type
one gets:
\[
b(N,\omega)\leq\sum_{n_{i}\notin
V}b(1,\tau^{n_{i+1}}\omega)+\sum_{n_{i}\in V}(b(n_{i},\omega)-b(n_{i+1},\omega)).
\]
Almost surely, $b(N,\omega)=o(N)$ when $N$ tends to infinity. In particular,
from a certain point onward, one has $b(N,\omega)\geq-N$. In the expression
on the right, we majorize the sum over $n_{i}\notin V$ by
$\sum_{j=0}^{N-1}b^{+}(1,\tau^{j}\omega)$ (where $b^{+}$ is the positive part
of $b$), which itself is bounded by $MN$ if $N$ is sufficiently large, where
we have set $M=1+\int b^{+}(1,\omega)$. Using the definition of $V$, we
obtain:
\[
  -N\leq MN -c\sum_{n_{i}\in V}(n_{i}-n_{i+1}).
\]
Since $V\cap[0,N-1]$ is included in the union of the intervals
$(n_{i+1},n_{i}]$ for $n_{i}\in V$, its cardinality $\Card{V\cap[0,N-1]}$ is
bounded by $\sum(n_{i}-n_{i+1})$. Therefore, the previous inequality gives
\[
  \Card{V\cap[0,N-1]}\leq(M+1)N/c.
\]
This finishes the proof, taking $c$ sufficiently large so that
$(M+1)/c\leq\delta$.
\end{proof}
In the following lemma, we replace $c$ (large) of Lemma~\ref{lem:warm-up} by
a parameter $\epsilon$ which is arbitrarily small. The price to pay in order
to preserve a valid statement is to restrict it to sufficiently large $\ell$.
This lemma plays a crucial role in the proof of Theorem~\ref{thm:subadd}.
\begin{lem}
\label{lem:crucial}Let $b$ be an integrable ergodic subadditive cocycle with
finite asymptotic average $A$. Let $\epsilon>0$ and $\delta>0$. There exists
$k\geq 1$ such that for almost every $\omega$
\[
\uppdensity\{n\in\mathbb{N}\st \exists\ell\in[k,n],\
b(n,\omega)-b(n-\ell,\omega)\leq(A-\epsilon)\ell\}\leq\delta.
\]
\end{lem}
\begin{proof}
Without loss of generality we may assume that $A=0$. We denote by $\tau$ the
underlying ergodic transformation. Going to the natural extension if
necessary, we can assume that $\tau$ is invertible.

The idea of the proof is that the argument we used to prove
Lemma~\ref{lem:warm-up} would work in our situation if $\int b^+(1,\omega)$
were small enough. This is not the case in general, but this is
asymptotically true for the iterates of the cocycle, by Kingman's theorem: if
$s$ is large enough, then $\int b^+(s,\omega)/s$ is very small. We will fix
such an $s$, discretize time to work in $s\mathbb{N}$, and follow the proof
of Lemma~\ref{lem:warm-up} in this set. Additional errors show up in the
approximation process, but they are negligible if $k$ in the statement of the
lemma is large enough. If one is to do this precisely, there is a problem
that $\tau^{s}$ is in general not necessarily ergodic. This issue is resolved
by working with times in the set $s\mathbb{N}+t$ for fixed $t\in [0, s-1]$.

Let us start the rigorous argument. Fix $\rho>0$, which corresponds to the
precision we want to achieve (this value will be chosen at the end of the
proof). Since $b(s,\omega)/s$ tends to $0$ almost everywhere and in $L^{1}$
when $s$ tends to infinity by Kingman's theorem, the same holds for $b^{+}$.
One can thus take $s\in\mathbb{N}$ such that $\int b^{+}(s,\omega)<\rho s$.
We also fix $t\in[0,s-1]$. Let $K=s\mathbb{N}+t$ be the set of reference
times we will use in the following. Once all these data are fixed, we take a
large enough $k$.

The set of bad times, whose density we want to majorize, can be decomposed as
$U\cup V$, where
\begin{align*}
U & =\{n\in\mathbb{N}\st \exists\ell\in(n-s,n],\ b(n,\omega)-b(n-\ell,\omega)\leq-\epsilon\ell\},\\
V & =\{n\in\mathbb{N}\st \exists\ell\in[k,n-s],\ b(n,\omega)-b(n-\ell,\omega)\leq-\epsilon\ell\}.
\end{align*}
If $n\in U$, then there exists $i\in[0,s)$ such that $b(n,\omega)\leq
b(i,\omega)-\epsilon(n-i)$. Since $b(n,\omega)/n\rightarrow0$ almost
everywhere, we deduce that $U$ is almost surely finite. It suffices therefore
to estimate the density of $V$.

Consider $n\in V$ and $\ell\in[k,n-s]$ such that
$b(n,\omega)-b(n-\ell,\omega)\leq-\epsilon\ell$. We will approximate such an
$n$ by a time in $K$. Let $\tilde{n}$ be the successor of $n$ in $K$, that is
the smallest time in $K$ with $\tilde{n}\geq n$. We write $\tilde{n}=n+i$
with $i<s$. Thus, \[ b(\tilde{n},\omega)=b(n+i,\omega)\leq
b(n,\omega)+b(i,\tau^{n}\omega)\leq b(n,\omega)+F(\tau^{\tilde{n}}\omega),\]
where we set $F(\eta)=\sum_{j=-s}^{s}b^{+}(1,\tau^{j}\eta)$, which is
integrable and positive. Similarly, as $n-\ell\geq s$ by assumption, $n-\ell$
admits a predecessor $\tilde{n}-\tilde{\ell}$ in $K$. One has
$n-\ell=\tilde{n}-\tilde{\ell}+j$ for some $j<s$, and
\[
  b(n-\ell,\omega)=b(\tilde{n}-\tilde{\ell}+j,\omega)
  \leq
  b(\tilde{n}-\tilde{\ell},\omega)+b(j,\tau^{\tilde{n}-\tilde{\ell}}\omega)
  \leq
  b(\tilde{n}-\tilde{\ell},\omega)+F(\tau^{\tilde{n}-\tilde{\ell}}\omega).
\]
We obtain
finally that
\begin{align*}
 b(\tilde{n},\omega)-b(\tilde{n}-\tilde{\ell},\omega) & \leq b(n,\omega)+F(\tau^{\tilde{n}}\omega)-b(n-\ell,\omega)+F(\tau^{\tilde{n}-\tilde{\ell}}\omega)\\
 & \leq-\epsilon\ell+F(\tau^{\tilde{n}}\omega)+F(\tau^{\tilde{n}-\tilde{\ell}}\omega)\\
 & \leq-\epsilon\tilde{\ell}/2+F(\tau^{\tilde{n}}\omega)+F(\tau^{\tilde{n}-\tilde{\ell}}\omega),
\end{align*}
where the last inequality comes from the fact that $\tilde{\ell}\leq\ell+2s$
is bounded by $2\ell$ whenever $k$ is sufficiently large.

Denote by
\[
  W=\{\tilde{n}\in K\st \exists\tilde{\ell}\in s\mathbb{N}\cap[k,\tilde{n}],
    \ b(\tilde{n},\omega)-b(\tilde{n}-\tilde{\ell},\omega)\leq-\epsilon\tilde{\ell}/2+F(\tau^{\tilde{n}}\omega)+F(\tau^{\tilde{n}-\tilde{\ell}}\omega)\}.
\]
We have shown that
\begin{equation}
\label{controle_V}
  V\subset W+[-s+1,0].
\end{equation}

To estimate the density of $V$, it suffices hence to estimate the density of
$W$. Let $N$ be an integer, let $\tilde{N}=ps+t$ be its successor in $K$ (it
satisfies $\tilde{N}\leq2N$ if $N$ is sufficiently large). We decompose
$K\cap[0,\tilde{N}]$ as in Lemma~\ref{lem:warm-up}. We start with
$\tilde{n}_{0}=\tilde{N}$. If we have defined $\tilde{n}_{i}$, we define its
predecessor as follows. If $\tilde{n}_{i}\notin W$, we take
$\tilde{n}_{i+1}=\tilde{n}_{i}-s$. If $\tilde{n}_{i}\in W$, then let
$\tilde{\ell}_{i}\in s\mathbb{N}\cap[k,\tilde{n}_{i}]$ as in the definition
of $W$, and set $\tilde{n}_{i+1}=\tilde{n}_{i}-\tilde{\ell}_{i}$. We stop
when $\tilde{n}_{i}=t$.

We have thus decomposed $[0,\tilde{N})$ as a union of intervals of the form
$[\tilde{n}_{i}-s,\tilde{n}_{i})$ (with $\tilde{n}_{i}\notin W$), and
$[\tilde{n}_{i}-\tilde{\ell}_{i},\tilde{n}_{i})$ (with $\tilde{n}_{i}\in W$)
and $[0,t)$.

Using the subadditivity along the intervals of the first and the third types,
one gets:
\[ b(\tilde{N},\omega)\leq b(t,\omega)+\sum_{\tilde{n}_{i}\notin
W}b(s,\tau^{\tilde{n}_{i+1}}\omega)+\sum_{\tilde{n}_{i}\in
W}(b(\tilde{n}_{i},\omega)-b(\tilde{n}_{i+1},\omega)).
\]
Almost surely, $b(\tilde{N},\omega)=o(\tilde{N})$ when $\tilde{N}$ tends to
infinity. Hence, after a certain stage, we have
$b(\tilde{N},\omega)\geq-\rho\tilde{N}$. In the terms on the right above, we
majorize the sum over $\tilde{n}_{i}\notin W$ by
$\sum_{j=0}^{p-1}b^{+}(s,\tau^{js+t}\omega)$. The trivial term $b(t,\omega)$
is estimated by $F(\omega)$. Using the definition of $W$, we obtain:
\[
-\rho\tilde{N}\leq
F(\omega)+\sum_{j=0}^{p-1}b^{+}(s,\tau^{js+t}\omega)+\sum_{\tilde{n}_{i}\in
W}(-\epsilon(\tilde{n}_{i}-\tilde{n}_{i+1})/2+F(\tau^{\tilde{n}_{i}}\omega)+F(\tau^{\tilde{n}_{i+1}}\omega)).
\]
The set $W$ is included in the union of the intervals
$(\tilde{n}_{i+1},\tilde{n}_{i}]$ with $\tilde{n}_{i}\in W$, by construction.
The same holds for $V$ in view of~\eqref{controle_V}. Therefore,
$\Card{V\cap[0,N-1]}\leq\sum(\tilde{n}_{i}-\tilde{n}_{i+1})$. The previous
equation gives hence that:
\[
\epsilon\Card{V\cap[0,N-1]}/2\leq\rho\tilde{N}+F(\omega)+\sum_{j=0}^{p-1}b^{+}(s,\tau^{js+t}\omega)+\sum_{\tilde{n}_{i}\in
W}(F(\tau^{\tilde{n}_{i}}\omega)+F(\tau^{\tilde{n}_{i+1}}\omega)).
\]
The function $F$ is integrable, so there exists $M\in\mathbb{R}$ such that
the function $G=F\cdot1_{F\geq M}$ satisfies $\int G<s\rho$. We majorize $F$
by $M+G$. In the preceding equation, the $\tilde{n}_{i}$s in $W$ are
separated by at least $k$ because $\tilde{\ell}_{i}\geq k$ by definition.
Therefore, the number of such $\tilde n_i$ is at most $\tilde{N}/k$. We get
\begin{align*}
\sum_{\tilde{n}_{i}\in W}(F(\tau^{\tilde{n}_{i}}\omega)+F(\tau^{\tilde{n}_{i+1}}\omega)) &
 \leq2(\tilde{N}/k)M+\sum_{\tilde{n}_{i}\in W}(G(\tau^{\tilde{n}_{i}}\omega)+G(\tau^{\tilde{n}_{i+1}}\omega))
 \\
 & \leq\rho\tilde{N}+2\sum_{j=0}^{p}G(\tau^{js+t}\omega),
\end{align*}
whenever $k$ is sufficiently big.

Finally, if $k$ is sufficiently large, one has (using that $\tilde{N}\leq2N$)
\[
\epsilon\Card{V\cap[0,N-1]}/2\leq F(\omega)+4\rho
N+\sum_{j=0}^{p}H(\tau^{js+t}\omega),
\]
where $H(\eta)=b^{+}(s,\eta)+2G(\eta)$ has integral $<3s\rho$. Summing after
that over $t\in[0,s-1]$, we obtain:
\[
s\epsilon\Card{V\cap[0,N-1]}/2\leq sF(\omega)+4s\rho N+\sum_{i=0}^{N+s-1}H(\tau^{i}\omega).
\]
For almost every $\omega$, Birkhoff's theorem applied to the function $H$
gives $\sum_{i=0}^{N+s-1}H(\tau^{i}\omega)\leq 3s\rho N$ for $N$ sufficiently
large. Thus
\[
s\epsilon\Card{V\cap[0,N-1]}/2\leq sF(\omega)+7s\rho N.
\]
This shows that the density of $V$ is bounded by $14\rho/\epsilon$. This
concludes the proof if we choose $\rho=\epsilon\delta/14$ at the beginning of
the argument.
\end{proof}
We combine the two preceding lemmas in order to gain improving control over
time, as follows.

\begin{lem}
\label{lem:alldensities}Let $b$ be an integrable ergodic subadditive cocycle
with finite asymptotic average $A$. Let $\epsilon>0$. There exist a sequence
$\delta_\ell\to 0$, a subset $O$ of measure at least $1-\epsilon$ and for
$\omega\in O$, there is a sequence of bad times $U(\omega)$ with
$\Card{U(\omega)\cap[0,n-1]}\leq\epsilon n$ for every $n$, with the following
property. For every $\omega\in O$, for all $n$ not in $U(\omega)$, and for
every $\ell\in[1,n]$, it holds that
\begin{equation}
  b(n,\omega)-b(n-\ell,\omega)\geq(A-\delta_\ell)\ell.\label{eq_controle_a}
\end{equation}
\end{lem}
\begin{proof}
For every $i>1$, set $c_{i}=2^{-i}$. In view of Lemma~\ref{lem:crucial},
there exists $k_{i}$ such that, for almost every $\omega$, the set
\[
U_{i}(\omega)=\{n\in\mathbb{N} \st \exists\ell\in[k_{i},n],\
b(n,\omega)-b(n-\ell,\omega)\leq(A-c_{i})\ell\}
\]
satisfies $\uppdensity\left\{ U_{i}(\omega)\right\} <\epsilon2^{-i}$. For $n$
large, say $n\geq n_{i}(\omega)$, one obtains
$\Card{U_{i}(\omega)\cap[0,n-1]}\leq\epsilon 2^{-i}n$. Since the function
$n_{i}(\omega)$ is almost everywhere finite, we may find a subset $O_{i}$ of
measure close to $1$, say $\mu(O_{i})>1-\epsilon2^{-i}$, on which
$n_{i}(\omega)\leq n_{i}$, for some integer $n_{i}$ (which one may choose to
be $>n_{i-1}$ and $\geq k_{i}$).

We treat the case $i=1$ separately and in a more crude manner, applying
Lemma~\ref{lem:warm-up}: there is a constant $c_{1}$ such that, for almost
every $\omega$,
\[
  \uppdensity\{n\in\mathbb{N}\st \exists\ell\in[1,n],\
  b(n,\omega)-b(n-\ell,\omega)\leq(A-c_{1})\ell\}< \epsilon/2.
\] We set
$k_{1}=1$. As above, we define hence $U_{1}(\omega)$, $n_{1}(\omega)$, and
$O_{1}$.

We set $\bar{O}=\bigcap_{i\geq1}O_{i}$, the good set on which things are well
controlled. It satisfies $\mu(\bar{O})>1-\epsilon$. For $\omega\in\bar{O}$ we
define a set of bad times $U(\omega)$ by \[
U(\omega)=\bigcup_{i\geq1}U_{i}(\omega)\cap[n_{i},+\infty).\]

We begin by showing that the bad set $U(\omega)$ satisfies
$\Card{U(\omega)\cap[0,n-1]}\leq\epsilon n$ for every $n$. Let
$n\in\mathbb{N}$. Let $i$ be such that $n_{i}\leq n<n_{i+1}$ (there is
nothing to do if $n<n_{1}$ since $U(\omega)\subset[n_{1},+\infty)$). Hence
\begin{align*}
\Card{U(\omega)\cap[0,n-1]} & \leq\Card*{\bigcup_{j\leq i}U_{j}(\omega)\cap[n_{j},+\infty)\cap[0,n-1]}
\leq\Card*{\bigcup_{j\leq i}U_{j}(\omega)\cap[0,n-1]}
\\
& \leq\sum_{j\leq i}\Card{U_{j}(\omega)\cap[0,n-1]}.
\end{align*}
Since $n\geq n_{j}$ for every $j\leq i$, the cardinality of
$U_{j}(\omega)\cap[0,n-1]$ is bounded by $\epsilon2^{-j}n$. Therefore the sum
is not greater than $\epsilon n$, as desired.

Set $I_{i}=[n_{i},n_{i+1})$ for $i>1$, and $I_{1}=[1,n_{2})$. We define a
sequence $\bar{\delta}_\ell=c_{i}$ for $\ell\in I_{i}$. This sequence tends
to $0$. We claim that it satisfies~\eqref{eq_controle_a} for every $n\geq
n_{1}$ which is not in $U(\omega)$. Indeed, fix $\ell\in[1,n]$, it belongs to
an interval $I_{i}$. We claim that $n\geq n_{i}$: This holds by assumption if
$i=1$, and if $i>1$ we have $n_{i}=\inf I_{i}\leq\ell\leq n$. As $n\geq
n_{i}$ and $n\notin U(\omega)$, we have $n\notin U_{i}(\omega)$. Moreover,
$\ell\geq k_{i}$: Indeed, if $i>1$, this follows from the inequalities
$\ell\geq n_{i}$ and $n_{i}\geq k_{i}$, while if $i=1$ this comes simply from
the fact that $k_{1}=1$. Thus, the definition of $U_{i}(\omega)$ ensures that
$b(n,\omega)-b(n-\ell,\omega)\geq(A-c_{i})\ell$, which gives the result since
$\bar{\delta}_\ell=c_{i}$.

This almost finishes the proof, it just remains to treat the case $n<n_{1}$
and $\ell\in[1,n]$. This is only a finite number of conditions. All the
functions we have considered are measurable, so they are almost everywhere
finite. There is therefore a subset $O$ of $\bar{O}$, again with
$\mu(O)>1-\epsilon$, and a constant $d$ such that for every $\omega\in O$,
and every $n<n_{1}$ and $\ell\in[1,n]$, we have
$b(n,\omega)-b(n-\ell,\omega)\geq(A-d)\ell$. Set finally
$\delta_\ell=\bar{\delta}_\ell$ for $\ell\geq n_{1}$ and
$\delta_\ell=\max(d,\bar{\delta}_\ell)$ for $\ell<n_{1}$. This function
works.
\end{proof}
We now deduce the theorem from these lemmas.

\begin{proof}[Proof of Theorem~\ref{thm:subadd}]
We may assume that the asymptotic average $A$ vanishes. First, we prove the
easy upper bound. By subadditivity,
\begin{equation*}
  a(n,\omega)-a(n-\ell, T^\ell \omega) \leq a(\ell, \omega),
\end{equation*}
which is almost surely $o(\ell)$ by Kingman's theorem. This proves the upper
bound in Theorem~\ref{thm:subadd}. The stronger statement in
Remark~\ref{rmk:subadd} follows from the fact that the almost sure
convergence $a(\ell,\omega)/\ell \to 0$ is uniform on sets of arbitrarily
large measure.

We turn to the harder lower bound. Let $b(n,\omega)=a(n,T^{-n}\omega)$. This
is a subadditive cocycle for the ergodic transformation $T^{-1}$. We may
therefore apply Lemma~\ref{lem:alldensities} to it. Let $\epsilon>0$. The
lemma gives us a set of good points $O$ with measure at least $1-\epsilon$, a
sequence $\delta_\ell\to0$ and, for $\omega\in O$, a set $U(\omega)$ of bad
times with $\Card{U(\omega)\cap[0,n-1]}\leq\epsilon n$ for every $n$.

Let $O_{n}=\{y\in O \st n \notin U(\omega)\}$ and $P_{n}=T^{-n}O_{n}$. For
$\omega\in P_{n}$ and $\ell\in[1,n]$, one has
\[
a(n,\omega)-a(n-\ell,T^{\ell}\omega)=b(n,T^{n}\omega)-b(n-\ell,T^{n}\omega)\geq -\delta_\ell \ell.
\]
Hence, if a point is contained in an infinite number of the sets $P_{n}$, it
satisfies the conclusion of the theorem. If the times where it belongs to
$P_{n}$ have an asymptotic density at least $1-\rho$, it satisfies even the
stronger conclusion in Remark~\ref{rmk:subadd}. We have to show that this
condition has large measure.

For $\omega\in\Omega$, we define $A(\omega)=\{n:\omega\in P_{n}\}$, its set
of good times. We would like to see that $A(\omega)$ has an upper asymptotic
density larger than $1-\rho$, for $\omega$ in a subset of large measure. Let
$f_{N}(\omega)=\Card{A(\omega)\cap[0,N-1]}$. The bad points are those for
which $f_{N}(\omega)\leq(1-\rho)N$ for all $N$ sufficiently large. Denote by
$V_{i}=\{\omega\st \forall N\geq i,\ f_{N}(\omega)\leq(1-\rho)N\}$, and
$V=\bigcup V_{i}$ the set of bad points.

We have
\begin{align*}
\int f_{N} & =\sum_{n=0}^{N-1}\mu(P_{n})=\sum_{n=0}^{N-1}\mu(O_{n})
 =\int 1_{O}(\omega)\Card{[0,N-1]\setminus U(\omega)}\dd\mu(\omega)\\
 & \geq\int1_{O}(\omega)(1-\epsilon)N\dd\mu(\omega)\geq(1-\epsilon)^{2}N.
\end{align*}

Since $f_{N}\leq N$, we obtain for $N>i$
\[
(1-\epsilon)^{2}N\leq\int f_{N}\leq(1-\rho)N\mu(V_{i})+N(1-\mu(V_{i}))
 =N - \rho N \mu(V_i).
\]
Thus $\mu(V_{i})\leq(1-(1-\epsilon)^2)/\rho$. We deduce that
$\mu(V)\leq(1-(1-\epsilon)^2)/\rho$. If we have chosen $\epsilon$
sufficiently small with respect to $\rho$, this quantity is bounded by
$\rho$. This proves that the lower bound of Theorem~\ref{thm:subadd} (and
even the stronger conclusion in Remark~\ref{rmk:subadd}) is satisfied on a
set of measure greater than $1-\rho$. Since $\rho$ is arbitrary the proof is
complete.
\end{proof}

Here is a small example showing that, even in deterministic situations, one
can not improve the lower bound in Theorem~\ref{thm:subadd} to a bound of the
form $a_n -a_{n-\ell} \geq A\ell - \delta_\ell$ (where $\delta_\ell$ is any
sequence tending to $0$) while keeping a lot of good times. Indeed, consider
a sequence $a_n$ which is either $1$ or $2$ for every $n$. This is always a
subadditive sequence. Assume also that the value $2$ is taken infinitely many
times. Then the times for which $a_n -a_{n-\ell} \geq -\delta_\ell$ for all
$\ell \leq n$ are, up to finitely many exceptions, only the times when
$a_n=2$. Hence, they can be arbitrarily sparse.

\section{Application to multiplicative ergodic theorems}
\label{sec:Application-to-multiplicative}

\begin{proof}[Proof of Theorem~\ref{thm:met}]
Let
\[
a(n,\omega)=d(u(n,\omega)x_{0},x_{0}).
\]
Since the maps are semi-contractive and thanks to the triangle inequality one
verifies easily that this is a subadditive ergodic cocycle with asymptotic
average $A\geq0$. In view of Theorem~\ref{thm:subadd} we have therefore for
almost every $\omega$ a sequence $n_{i}\rightarrow\infty$ and a sequence
$\delta_\ell\to 0$ such that for every $i$ and every $\ell\leq n_{i}$,
\[
d(u(n_{i},\omega)x_{0},x_{0})-d(u(n_i-\ell,T^{\ell}\omega)x_{0},x_{0})\geq(A-\delta_{\ell})\ell.
\]
If we write $x_{n}=u(n,\omega)x_{0}$ and $h_n$ the horofunction associated to
$x_n$, this means that
\begin{align*}
h_{n_i}(x_\ell) = d(x_{n_{i}},x_{\ell})-d(x_{n_i},x_{0})&=d(u(n_{i},\omega)x_{0},u(\ell,\omega)x_{0})-d(u(n_i,\omega)x_{0},x_{0})
\\& \leq
d(u(n_{i}-\ell,T^{\ell}\omega)x_{0},x_{0})-d(u(n_{i},\omega)x_{0},x_{0})
\\& \leq-(A-\delta_{\ell})\ell.
\end{align*}
This inequality passes to limits as $n_{i}\rightarrow\infty$.

If $X$ is separable, one may extract a subsequence of $h_{n_i}$ such that
$h_{n'_i}(y)$ converges for all $y$ belonging to a countable dense set of
$X$. Since all these functions are $1$-Lipschitz, convergence on every point
of $X$ follows. The limit $h$ of $h_{n'_i}$ satisfies for all $\ell$ the
inequality
\begin{equation}
\label{eq:h_limit_ineq}
  h(x_{\ell})\leq-(A-\delta_{\ell})\ell.
\end{equation}

In the general case, $\overline{\Phi(X)}$ is still compact, but it does not
have to be sequentially compact, so we should argue differently. The sets
\[
  Y_i = \{h \in \overline{\Phi(X)} \st \forall \ell \leq n_i,\ h(x_{\ell})\leq-(A-\delta_{\ell})\ell\}
\]
are non-empty (they contain $h_{n_i}$) and form a decreasing sequence. By
compactness, $\bigcap_i Y_i$ is also non-empty. Any element $h$ of this
intersection satisfies~\eqref{eq:h_limit_ineq}.

The bound $\abs{h(x_{\ell})}\leq d(x_{\ell},x_{0})\leq A\ell+o(\ell)$ follows
from the 1-Lipschitz property of $h$ and Kingman's theorem. Therefore
\[
\lim_{n\rightarrow\infty}-\frac{1}{n}h(u(n,\omega)x_{0})=A
\]
as required.

Finally, let us show that $\omega \mapsto h^\omega$ can be chosen to be Borel
measurable if $\Omega$ is standard and $X$ is separable. In this case, the
topology on $\overline{\Phi(X)}$ is generated by simple convergence along a
dense sequence in $X$. Hence, $\overline{\Phi(X)}$ is metrizable, and it
becomes a compact metric space. Since $\Omega$ is standard, we may identify
it with $[0,1]$.

By Remark~\ref{rmk:subadd}, there exists a decomposition of $\Omega$ as the
union of a set $\Omega_\infty$ of measure $0$, and an increasing sequence of
sets $\Omega_i$ on which one can use the same sequence $\delta_{i,\ell}$. By
Lusin's theorem, we may also ensure that all the maps $\omega \mapsto
u(n,\omega)x_0$ are continuous on $\Omega_i$. Let $\Lambda_i=\Omega_i
\setminus \Omega_{i-1}$, it suffices to find a Borel map $\omega \mapsto
h^\omega$ on each $\Lambda_i$. Define
\[
  A(\omega)\coloneqq\{ h\in \overline{\Phi(X)} \st \forall \ell\in \mathbb{N},\ h(x_\ell(\omega)) \leq -(A-\delta_{i,\ell})\ell\}.
\]
This is a nonempty compact subset of $\overline{\Phi(X)}$, depending upper
semi-continuously on $\omega\in \Lambda_i$. We are looking for a measurable
map $\omega \in \Lambda_i \mapsto h^\omega \in A(\omega)$, the only
difficulty being measurability. Its existence follows for instance
from~\cite[Theorem 4.1]{W77}.
\end{proof}

One can deduce Corollary~\ref{cor:Banach} as a consequence of
Theorem~\ref{thm:met} together with the following lemma.

\begin{lem}
In a Banach space, for every metric functional $h$ there is a linear
functional $f$ of norm at most $1$ such that $f\leq h$.
\end{lem}
\begin{proof}
Let $h$ be a metric functional on a Banach space $X$. We will show that, for
any finite set $F$ in $X$ and any $\epsilon>0$, there is a linear form $f$ on
$X$ with norm at most $1$ such that $f(x) \leq h(x)+\epsilon$ for all $x\in
F$. By weak{*} compactness (the Banach-Alaoglu theorem), the existence of a
linear form $f$ with norm at most $1$ such that $f(x)\leq h(x)$ for all $x\in
X$ follows.

The open set
\begin{equation*}
  \{ \tilde h \in \overline{\Phi(X)} \st \forall x\in F,\ \tilde h(x) < h(x)+\epsilon\}
\end{equation*}
is nonempty (it contains $h$). Hence, as $\overline{\Phi(X)}$ is the closure
of $X$, this set contains a point $y \in X$. It satisfies for all $x\in F$
\begin{equation*}
  \norm{y-x} - \norm{y} = h_y(x) < h(x)+\epsilon.
\end{equation*}
By Hahn-Banach theorem, there exists a linear form $f$ of norm $1$ with $f(y)
=-\norm{y}$. Then, for all $x\in F$,
\begin{equation*}
  f(x) = f(x-y) + f(y) \leq \norm{x-y} -\norm{y} \leq h(x)+\epsilon,
\end{equation*}
as desired.
\end{proof}

Alternatively, we give a direct proof of Corollary~\ref{cor:Banach} within
the vector space setting without referring to metric functionals:

\begin{proof}[Proof of Corollary~\ref{cor:Banach} directly from Theorem~\ref{thm:subadd}]
Let \[ a(n,\omega)=\norm{ u(n,\omega)0} ,\] which is a subadditive ergodic
cocycle with asymptotic average $A\geq0$. If $A=0$, then the conclusion is
trivial so we assume that $A>0$. In view of Theorem~\ref{thm:subadd} we have
therefore for almost every $\omega$ a sequence $n_{i}\rightarrow\infty$ and a
sequence $\delta_\ell\to 0$ such that for every $i$ and every $\ell\leq
n_{i}$,
\[
\norm{ u(n_{i},\omega)0} -\norm{ u(n_i-\ell,T^{\ell}\omega)0}
\geq(A-\delta_{\ell})\ell.
\]
We denote by $x_{n}=u(n,\omega)0$. By the Hahn-Banach theorem we can find a
linear functional of norm $1$ such that $f_{n}(x_{n})=\norm{ x_{n}}$. Now,
for any $\ell\leq n_i$,
\begin{align*}
f_{n_{i}}(x_{\ell})& =f_{n_{i}}(x_{n_{i}}+x_{\ell}-x_{n_{i}})
=\norm{u(n_{i},\omega)0} -f_{n_i}(x_{n_i}-x_{\ell})
\\&
  \geq\norm{ u(n_{i},\omega)0} -\norm{u(n_{i},\omega)0-u(\ell,\omega)0}
 \geq \norm{u(n_{i},\omega)0} -\norm{u(n_{i}-\ell,T^{\ell}\omega)0-0}
\\&
\geq(A-\delta_{\ell})\ell.
\end{align*}
By weak{*} compactness, there exists a linear functional $f=f^{\omega}$ of
norm at most $1$ satisfying
\[
  f(x_\ell) \geq (A-\delta_{\ell})\ell
\]
for all $\ell\geq 0$. It follows that
\[
\lim_{\ell\rightarrow\infty}\frac{1}{\ell}f(x_{\ell})=A
\]
a.e.\ as required. In the case $A>0$, the norm of $f$
is clearly necessarily $1$.
\end{proof}
This has in turn another consequence, Corollary~\ref{cor:randommeanergodic},
as is explained in the introduction.

\section{Cocycles of bounded linear operators}
\label{sec:Cocylces-of-bounded}

Invertible $d\times d$ matrices act by isometry on the symmetric space
$\Pos_{d}=GL_{d}/O_{d}$. How Theorem~\ref{thm:met} in this special case
implies Oseledets theorem is explained for example in~\cite{KM99,K04}.
Similarly, bounded linear invertible operators of a Hilbert space $H$ act by
isometry on the space $\Pos(H)$ of the positive elements of the algebra
$B(H)$ of all bounded linear operators $H\rightarrow H$. It is a cone in the
vector space $\Sym(H)=\left\{ a\in B(H):a^{*}=a\right\} $. Thus for
integrable ergodic cocycles of bounded linear operators we may again apply
Theorem~\ref{thm:met}. In contrast to the finite dimensional case, the metric
of $\Pos$ is less nice and the space is not locally compact. Therefore the
metric functionals are less studied at present time. An alternative approach
is provided thanks to Segal's inequality
\[
\norm{ e^{u+v}} \leq\norm{e^{u/2}e^{v}e^{u/2}} .
\]
This implies a weak notion of non-positive curvature: the differential
isomorphism $\exp:\Sym\rightarrow \Pos$ semi-expands distances. This means
that the inverse, the logarithm, is semi-contractive.

So let
$v(n,\omega)=\varphi(T^{n-1}\omega)\varphi(T^{n-2}\omega)\dotsm\varphi(\omega)$
be an integrable ergodic cocycle of bounded invertible linear operators of
$H$. Note that we here take the opposite order compared to~\eqref{eq:def_u}.
Hence if we denote $p_{n}$ the positive part of $v(n,\omega)$, that is
\[
p_{n}(\omega)=\left(v(n,\omega)^{*}v(n,\omega)\right)^{1/2},
\]
then $a(n,\omega)\coloneqq \norm{ \log p_{n}(\omega)} $ is a subadditive
cocycle, where the norm is the operator norm. Indeed, notice that the $p_{n}$
is the orbit of the matrices $v(n,\omega)^{*}$ acting by isometry on $\Pos$,
now in the right order for the metric statements. The distance from the identity to the $n$th point of the random orbit in $\Pos$
gives a subadditive cocycle. Thanks to that logarithm preserves distance from
Id to $p$, and contracts distances between $p$ and $q$ not the identity,
the inequalities between distances go in the right direction, so that the
proof of Corollary~\ref{cor:Banach} as given in
Section~\ref{sec:Application-to-multiplicative} goes through. We conclude
that for a.e.\ $\omega$ there is a linear functional $F_{\omega}$ on $\Sym$
(or on the full space of bounded linear operators, by Hahn-Banach) of norm 1 such that
\[
\lim_{n\rightarrow\infty}\frac{1}{n}F_{\omega}\left(\log
p_{n}(\omega)\right)=\lim_{n\rightarrow\infty}\frac{1}{n}\norm{ \log
p_{n}(\omega)} ,
\]
which is Theorem~\ref{thm:operators}.

For comparison, the classical formulation of the multiplicative ergodic
theorem is equivalent to the fact that $p_{n}(\omega)^{1/n}$, or
$\frac{1}{n}\log p_{n}(\omega)$, converges in norm as $n\rightarrow\infty$.
In general $\Sym$ is not uniformly convex, so our weaker convergence seems
near best possible. Under stronger assumptions, one can probably promote it
to a stronger statement. When $\Sym$ is a Hilbert space, for example in the
setting of Hilbert-Schmidt operators, the linear functionals are given by
$M\mapsto \Tr(A M)$. This implies Oseledets theorem (it is actually stronger,
a more uniform convergence) as Ruelle explains for compact operators,
see~\cite{KM99}.

As was remarked in the introduction, it is well-known that in general, even
for products of just one operator, one cannot hope for a Oseledets-type
Lyapunov regularity.


%

\section{Cocycles of holomorphic maps}
\label{sec:Cocycles-of-holomorphic}

Pseudo-metrics are frequently employed in the theory of several complex
variables, partly because of the Schwartz lemma but also thanks to their
connection to diophantine problems (Lang's conjecture). Given a complex space $Z$, we denote the
associated Kobayashi pseudo-distance $d_{Z}$. Every holomorphic map between
complex spaces $f:Z\rightarrow W$ is 1-Lipschitz with respect to these
pseudo-metrics:\[ d_{W}(f(z_{1}),f(z_{2}))\leq d_{Z}(z_{1},z_{2})\] for all
points $z_{1},z_{2}\in Z$. \emph{Pseudo} refers to that for these distances
the axiom about $d(x,y)=0$ may fail. On one extreme we have that the
pseudo-metric on $\mathbb{C}$ is identically $0$ for all pairs of points,
while for hyperbolic Riemann surfaces it is always non-degenerate, and thus
an honest metric. (These facts already explain the theorems of Liouville and
Picard on entire functions.)

Many papers have been devoted to the topic of extending the Wolff-Denjoy
theorem, and there are also papers about composing random maps, in both
orders, generalizing continued fraction expansion, see~\cite{KeL07} and
references therein.

Even for a pseudo-metric one defines metric functionals and horofunctions as before. So our
multiplicative ergodic theorem in principle applies, and gives an extension
of the Wolff-Denjoy theorem to a vastly more general situation:
\begin{thm}
Let $u(n,\omega)$ be an integrable ergodic cocycle of holomorphic self-maps
of a complex space $Z$. Then for a.e.\ $x$ there is a metric functional
$h^{\omega}$ for the pseudo-metric space $(Z,d_{Z})$ such that
\[
\lim_{n\rightarrow\infty}-\frac{1}{n}h^{\omega}(u(n,\omega)z)=\tau\coloneqq
\lim_{n\rightarrow\infty}\frac{1}{n}d_Z(u(n,\omega)z,z).
\]

\end{thm}
It remains to understand horofunctions. Under certain convexity and
smoothness assumptions, these metrics are Gromov hyperbolic or something
slightly weaker, and our result then implies that the orbit converges to a
boundary point, provided $\tau>0$. For the state-of-the-art of the metric
geometry of the Kobayashi metric, we refer to~\cite{K05,AR14,Z14} and
references therein. Here is a corollary:
\begin{cor}
Let $u(n,\omega)$ be an integrable ergodic cocycle of holomorphic maps of
$D$, where $D$ is a bounded domain in $\mathbb{C}^{d}$ which is either
strictly convex, strictly pseudo-convex with $C^{2}$-smooth boundary, or
pseudo-convex with analytic boundary. Unless for a.e.~$\omega$
\[
\frac{1}{n}d_D(u(n,\omega)z,z)\rightarrow 0
\]
as $n\rightarrow\infty$, it holds that a.e.\ orbit
$u(n,\omega)z$ converges to some boundary point $\xi_{\omega}\in\partial D$.
The boundary point may depend on $\omega$ but is independent of $z\in D$.
\end{cor}
\begin{proof}
It is known that under these assumptions on $D$ the metric space $(D,d)$ is a
proper metric space, where $d=d_{D}$ the Kobayashi metric. Assuming $\tau>0$,
the orbit accumulates on $\partial D$ and our Theorem~\ref{thm:met} provides
for a.e.\ $\omega$ a horofunction $h$ such that \[
h(u(n,\omega)z)\rightarrow-\infty\] when $n\rightarrow\infty$ and any $z\in
D$. We may assume that a sequence $x_{n}$ defining $h$ (say with base point
$x$) converges to some point $\xi$ in $\partial D.$

In the case $D$ is strictly convex it is shown in~\cite{AR14} that Abate's
big horospheres
\[
F_{x}(\xi,R)=\left\{ z\in D \st
\liminf_{w\rightarrow\xi}d(z,w)-d(x,w)<\frac{1}{2}\log R\right\}
\]
can only meet the boundary in one point. It is clear that $\left\{ z \st
h(z)<0\right\} $ is contained in $F_{x}(\xi,0)$. Thus we must have that
$u(n,\omega)z\rightarrow\xi$ as $n\rightarrow\infty$.

In the two remaining cases it is known from~\cite{K05} and references therein
that for any sequence $z_{n}$ converging to a different boundary point, there
is a constant $R>0$ such that for all $n,m$
\[ R>(z_{m},x_{n})\coloneqq
\frac{1}{2}\left(d(x,z_{m})+d(x,x_{n})-d(z_{m},x_{n})\right).
\]
Therefore it
would be impossible that $h(z_{m})<0$ since $d(x,z_{m})\rightarrow\infty$.
Hence again the conclusion that $u(n,\omega)z\rightarrow\xi$.
\end{proof}

\section{Behaviour of extremal length under holomorphic self-maps of Teichm\"uller
space} \label{sec:Behaviour-of-extremal}

Thurston announced in his celebrated preprint from 1976~\cite[Theorem 5]{T88}
that isotopy classes of surface diffeomorphisms admit some kind of Lyapunov
exponents. Let $S$ be a closed surface of genus $g\geq 2$. For any isotopy
class $\alpha$ of simple closed curves on $S$, and $\rho$ a Riemannian
metric, the length $L_{\rho}(\alpha)$ is the shortest length of a curve in
the isotopy class $ \alpha$ for the metric $\rho$. Given a diffeomorphism $f$
of $S$, there are a finite number of exponents $\lambda_i$ such that
\[
L_{\rho}(f^{n}\alpha)^{1/n}\rightarrow\lambda_i
\]
as $n\rightarrow\infty$ for some $i$ depending on $\alpha$. In the generic case there is
only one exponent. This is proved passing to the Teichm\"uller space $T_{g}$ and
instead letting $f$ act on the equivalence classes of metrics, which are
points in $T_{g}$. Indeed, one has
\[
L_{\rho}(f^{n}\alpha) =
L_{f^{-n}\rho}(\alpha).
\]

This was partly generalized to cocycles in~\cite{K14}, see also~\cite{H14}
for a refinement in the i.i.d.\ case. In several instances, again mainly due
to Thurston, certain holomorphic self-maps of the Teichm\"uller space arise.
Unless they are biholomorphic they do not give rise to an isotopy class of
diffeomorphisms of the underlying surface. It is therefore natural to
consider how lengths behave under the metric $u(n,\omega)$ with this order of
composition. In the case of holomorphic maps it is more natural to consider
length from complex analysis, namely the extremal length of
Beurling-Ahlfors:
\[
\Ext_{x}(\alpha)=\sup_{\rho\in\left[x\right]}\frac{L_{\rho}(\alpha)^{2}}{Area(\rho)},
\]
where the supremum is taken over all metrics in the conformal class of $x$.
We get applying our main theorem to the Teichm\"uller distance, using the
identification of horofunctions in this metric due to Liu and Su, and
following the arguments in~\cite{K14}:
\begin{thm}
Let $u(n,\omega)$ be an integrable cocycle of holomorphic self-maps of the
genus~$g$ Teichm\"uller space $T_{g}$. Denote by $d_{T}$ the Teichm\"uller
distance. Then for a.e.~$\omega$ there is a simple closed curve
$\alpha=\alpha_{\omega}$ such that
\[ \lim_{n\rightarrow\infty}\frac{1}{n}\log
\Ext_{u(n,\omega)\rho}(\alpha)
=2\lim_{n\rightarrow\infty}\frac{1}{n}d_{T}(u(n,\omega)\rho,\rho).
\]
\end{thm}
The link between the Teichm\"uller metric $d_{T}$ and extremal length comes via
Kerckhoff's formula:
\begin{equation*}
  d_T(x,y) = \sup_{\alpha} \frac{1}{2} \log \frac{\Ext_x(\alpha)}{\Ext_y(\alpha)}.
\end{equation*}


\begin{thebibliography}{KeL07}

\bibitem[ABV00]{ABV00}Alves, Jos\'e F.; Bonatti, Christian;
    Viana, Marcela SRB measures for partially hyperbolic systems whose central
    direction is mostly expanding. Invent. Math. 140 (2000),
    351\textendash{}398.

\bibitem[AR11]{AR11}Abate, Marco; Raissy, Jasmin Backward iteration
in strongly convex domains. Adv. Math. 228 (2011), no. 5, 2837\textendash{}2854.

\bibitem[AR14]{AR14}Abate, Marco; Raissy, Jasmin Wolff-Denjoy theorems
in nonsmooth convex domains. Ann. Mat. Pura Appl. (4) 193 (2014),
no. 5, 1503\textendash{}1518.

\bibitem[BS57]{BS57}Beck, Anatole; Schwartz, J. T. A vector-valued
random ergodic theorem. Proc. Amer. Math. Soc. 8 1957 1049\textendash{}1059.

\bibitem[Co88]{Co88}Cohen, Joel E. Subadditivity, generalized products
of random matrices and operations research. SIAM Rev. 30 (1988), no.
1, 69\textendash{}86.

\bibitem[CT80]{CT80}Crandall, Michael G.; Tartar, Luc Some relations
between nonexpansive and order preserving mappings. Proc. Amer. Math.
Soc. 78 (1980), no. 3, 385\textendash{}390.

\bibitem[ER85]{ER85}Eckmann, J.-P.; Ruelle, D. Ergodic theory of
chaos and strange attractors. Rev. Modern Phys. 57 (1985), no. 3,
part 1, 617\textendash{}656.

\bibitem[F02]{F02}Forni, Giovanni Deviation of ergodic averages for
area-preserving flows on surfaces of higher genus. Ann. of Math. (2)
155 (2002), no. 1, 1\textendash{}103.

\bibitem[GV12]{GV12}Gaubert, St\'ephane; Vigeral, Guillaume A maximin
characterisation of the escape rate of non-expansive mappings in metrically
convex spaces. Math. Proc. Cambridge Philos. Soc. 152 (2012), no.
2, 341\textendash{}363.

\bibitem[GG14]{GG14}Genady Ya. Grabarnik, Misha Guysinsky, Liv\v{s}ic Theorem
    for Banach Rings, arXiv:1408.5639 {[}math.DS{]}, 2014

\bibitem[Gu03]{Gu03}Gunawardena, Jeremy From max-plus algebra to
nonexpansive mappings: a nonlinear theory for discrete event systems.
Max-plus algebras. Theoret. Comput. Sci. 293 (2003), no. 1, 141\textendash{}167.

\bibitem[H14]{H14}Horbez, Camille The horoboundary of outer space, and growth
    under random automorphisms. arXiv:1407.3608 {[}math.GR{]}

\bibitem[Ka52]{Ka52}Kakutani, Shizuo Ergodic theory. Proceedings
of the International Congress of Mathematicians, Cambridge, Mass.,
1950, vol. 2, pp. 128\textendash{}142. Amer. Math. Soc., Providence,
R. I., 1952..

\bibitem[Ka11]{Ka11}B. Kalinin, Liv\v{ }sic Theorem for matrix cocycles
, Annals of Mathematics 173 (2011), 1025\textendash{} 1042.

\bibitem[KM99]{KM99}Karlsson, Anders; Margulis, Gregory A. A multiplicative
ergodic theorem and nonpositively curved spaces. Comm. Math. Phys.
208 (1999), no. 1, 107\textendash{}123.

\bibitem[K01]{K01}Karlsson, Anders Non-expanding maps and Busemann
functions. Ergodic Theory Dynam. Systems 21 (2001), no. 5, 1447\textendash{}1457.

\bibitem[K02]{K02}Karlsson, Anders Nonexpanding maps, Busemann functions,
and multiplicative ergodic theory. Rigidity in dynamics and geometry
(Cambridge, 2000), 283\textendash{}294, Springer, Berlin, 2002

\bibitem[K04]{K04}Karlsson, Anders Linear rate of escape and convergence
in direction. Random walks and geometry, 459\textendash{}471, Walter
de Gruyter, Berlin, 2004.

\bibitem[K05]{K05}Karlsson, Anders, On the dynamics of isometries,
Geometry and Topology, 9 (2005) 2359-2394

\bibitem[KL06]{KL06}Karlsson, Anders; Ledrappier, Fran\c{c}ois On laws
of large numbers for random walks. Ann. Probab. 34 (2006), no. 5,
1693\textendash{}1706.

\bibitem[K14]{K14}Karlsson, Anders Two extensions of Thurston's spectral
theorem for surface diffeomorphisms. Bull. Lond. Math. Soc. 46 (2014),
no. 2, 217\textendash{}226

\bibitem[KeL07]{KeL07}Keen, Linda; Lakic, Nikola Hyperbolic geometry from a
    local viewpoint. London Mathematical Society Student Texts, 68. Cambridge
    University Press, Cambridge, 2007. x+271 pp. ISBN: 978-0-521-68224-4;
    0-521-68224-X 30-01

\bibitem[Ki68]{Ki68}Kingman, J. F. C. The ergodic theory of subadditive
stochastic processes. J. Roy. Statist. Soc. Ser. B 30 1968 499\textendash{}510.

\bibitem[KN81]{KN81}Kohlberg, Elon; Neyman, Abraham Asymptotic behavior
of nonexpansive mappings in normed linear spaces. Israel J. Math.
38 (1981), no. 4, 269\textendash{}275.

\bibitem[LL10]{LL10}Lian, Zeng; Lu, Kening Lyapunov exponents and
invariant manifolds for random dynamical systems in a Banach space.
Mem. Amer. Math. Soc. 206 (2010), no. 967, vi+106 pp.

\bibitem[LY12]{LY12}Lian, Zeng; Young, Lai-Sang Lyapunov exponents,
periodic orbits, and horseshoes for semiflows on Hilbert spaces. J.
Amer. Math. Soc. 25 (2012), no. 3, 637\textendash{}665.

\bibitem[Ma97]{Ma97}Mairesse, Jean Products of irreducible random
matrices in the (max,+) algebra. Adv. in Appl. Probab. 29 (1997),
no. 2, 444\textendash{}477.

\bibitem[PR83]{PR83} Plant, Andrew T.; Reich, Simeon The asymptotics of nonexpansive iterations.
J. Funct. Anal. 54 (1983), no. 3, 308\textendash{}319.

\bibitem[R82]{R82}Ruelle, David Characteristic exponents and invariant
manifolds in Hilbert space. Ann. of Math. (2) 115 (1982), no. 2, 243\textendash{}290.

\bibitem[R84]{R84}Ruelle, David Characteristic exponents for a viscous
fluid subjected to time dependent forces. Comm. Math. Phys. 93 (1984),
no. 3, 285\textendash{}300

\bibitem [Sc06]{Sc06} Schauml\"offel, Kay-Uwe, Multiplicative ergodic theorems in infinite dimensions,
In: Infinite-Dimensional Random Dynamical Systems, Lyapunov Exponents,
Springer Lecture Notes in Mathematics, volume 1486, (2006) pp 187\textendash{}195

\bibitem[S89]{S89}Steele, J. Michael Kingman's subadditive ergodic
theorem. Ann. Inst. H. Poincar\'e Probab. Statist. 25 (1989), no. 1,
93\textendash{}98.

\bibitem[T88]{T88}Thurston, William P. On the geometry and dynamics
of diffeomorphisms of surfaces. Bull. Amer. Math. Soc. (N.S.) 19 (1988),
no. 2, 417\textendash{}431.

\bibitem[W77]{W77} Wagner, Daniel H. Survey of measurable selection theorems.
    SIAM J. Control Optimization 15 (1977), no. 5, 859\textendash{}903.

\bibitem[Wa08]{Wa08}Walsh, Cormac The horofunction boundary of the Hilbert geometry. Adv. Geom. 8 (2008), no. 4, 503\textendash{}529.

\bibitem[Z14]{Z14}Zimmer, Andrew M., Gromov hyperbolicity, the Kobayashi
    metric, and $C$-convex sets, preprint.


\end{thebibliography}
\end{document}